\documentclass[leqno]{amsart}

\usepackage[english]{babel} 									%Support for many languages
\usepackage[T1]{fontenc}							%The last two are related to making "ä" etc. more usable in Latex
\usepackage[utf8]{inputenx}
				
\usepackage{mathtools}								%Includes amsmath, basic math commands
\usepackage{empheq}									%Useful for systems of equations
\usepackage{enumitem}								%Improves enumerate environment among other things

\usepackage{amssymb}								%Fonts and mathematical symbols
\usepackage{mathrsfs}								%\mathscr command; http://www.math.washington.edu/~lee/Courses/typesetting-script.pdf
\usepackage{bm}                                     % Bold all the things in math mode 
				
\usepackage{verbatim}                               %comment environment: hide everything
\usepackage{iflang}									%Used in theoremenv.sty
\usepackage{xifthen}								%See http://tex.stackexchange.com/questions/58628/optional-argument-for-newcommand
\usepackage[final]{pdfpages}						%Enables including multiple pages of a pdf easily.
\usepackage{todonotes}								%\todo{note} \usepackage[disable]{todonotes}
\usepackage{aliascnt}								%Useful for making a counter share the number with other counter, but with different names.. used in thereomenv.sty in amsthm 																								 variants of the code
\usepackage{needspace}								%Needed in theoremenv, checks that proofs etc. don't get separated too easily: http://tex.stackexchange.com/a/30125
\usepackage{refcount}

\usepackage{subcaption}			%Naming figures (subfigure)
\usepackage{placeins} 			%Restricting movement of floats (FloatBarrier)
\usepackage{grffile}			%More possibilities with naming figures
\usepackage[all]{xy}			%Useful in Algebraic Topology

\usepackage{transparent}
\usepackage{color}				%Colorful pictures

\usepackage{tikz}               % For drawing diagrams.
\usetikzlibrary{math,arrows.meta, cd}
\DeclareGraphicsExtensions{.pdf,.png,.jpg}

\usepackage{graphicx}			%More options for includegraphic. See http://ctan.org/pkg/graphicx for example.
 
\usepackage{diagbox} %Diagonals to tableau

\usepackage[babel]{csquotes}							%More options for \cite etc.					

\usepackage{nameref}									%I.e. can refer to Theorem (Euler's theorem) as "Euler's theorem" with a hyperlink
\usepackage[colorlinks	=	true,						%Hyperlinks to references (an absolute must have)
          	 	linkcolor	=	red,
            	urlcolor	=	red,
            	citecolor	=	red]%
            	{hyperref}
\usepackage{bookmark}									%Hyperref options; used in theoremenv to bookmark theorems in pdf files automatically
\usepackage{cleveref}									%Further options for hyperref, \cref and \Cref are interesting (and their multiple variants)
\numberwithin{equation}{section}

\newtheorem{theorem}{Theorem}[section]

\newtheorem{corollary}[theorem]{Corollary}
\newtheorem{lemma}[theorem]{Lemma}

\newtheorem*{theorem*}{Theorem}

{\bf}{\it}

\theoremstyle{remark}
\newtheorem{remark}[theorem]{Remark}

\newtheorem*{remark*}{Remark}

\theoremstyle{remark}

\DeclareMathOperator{\diam}{diam}

\allowdisplaybreaks

\newcommand{\apmd}[2][]{							%Approximate Metric Derivative with options
	\ifthenelse{\equal{#1}{}}%
					{ \operatorname{N}_{#2}	}%
					{ \operatorname{N}_{#1,#2} 	}}

\newcommand{\aint}[2][]{%       %Integral average ($\aint[a]{b} corresponds to $\int_{a}^{b}$ with "bar" and $\aint{X}$ to $\int_{X}$ with bar.
	\ifthenelse{\equal{#1}{}}%
					{%
\mathchoice%
      {\mathop{\kern 0.2em\vrule width 0.6em height 0.69678ex depth -0.58065ex
              \kern -0.8em \intop}\nolimits_{\kern -0.45em#2}^{#1}}%
      {\mathop{\kern 0.1em\vrule width 0.5em height 0.69678ex depth -0.60387ex
              \kern -0.6em \intop}\nolimits_{#2}^{#1}}%
      {\mathop{\kern 0.1em\vrule width 0.5em height 0.69678ex depth -0.60387ex
              \kern -0.6em \intop}\nolimits_{#2}^{#1}}%
      {\mathop{\kern 0.1em\vrule width 0.5em height 0.69678ex depth -0.60387ex
              \kern -0.6em \intop}\nolimits_{#2}^{#1}}}%
					{%
\mathchoice%
      {\mathop{\kern 0.2em\vrule width 0.6em height 0.69678ex depth -0.58065ex                                              
              \kern -0.8em \intop}\nolimits_{\kern -0.45em#1}^{#2}}%
      {\mathop{\kern 0.1em\vrule width 0.5em height 0.69678ex depth -0.60387ex
              \kern -0.6em \intop}\nolimits_{#1}^{#2}}%
      {\mathop{\kern 0.1em\vrule width 0.5em height 0.69678ex depth -0.60387ex
              \kern -0.6em \intop}\nolimits_{#1}^{#2}}%
      {\mathop{\kern 0.1em\vrule width 0.5em height 0.69678ex depth -0.60387ex
              \kern -0.6em \intop}\nolimits_{#1}^{#2}}}}

\newcommand{\vol}{\mathrm{vol}}

\usepackage{thmtools}
\usepackage{thm-restate}

\begin{document}
\title[Growth of entire conformal curves]{Entire conformal curves are affine or have super-Euclidean energy growth}

\author{Toni Ikonen}

\address{Department of Mathematics\\ University of Fribourg\\  Chemin du Mus\'ee 23\\  1700 Fribourg, Switzerland}

\email{toni.ikonen@unifr.ch}

\thanks{The author was supported by the Swiss National Science Foundation grant 212867.}

\begin{abstract}
   We prove that entire conformal curves $\mathbb{R}^n \rightarrow \mathbb{R}^m$ fall into two classes: either the curve is affine or the average energy in a ball is strictly increasing for large radii and diverges to infinity. This rigidity follows from a blow-down argument and the strong interaction of the generalized Cauchy--Riemann equations with calibrated geometries and the sharp isoperimetric inequality for integral currents due to Almgren. As an application, we prove that every entire Lipschitz conformal curve is affine. This can be considered a higher-dimensional analog of the statement that an entire holomorphic function with a bounded complex differential is affine.

   We also recall that a certain punctured cone over a Legendrian torus provides a submanifold in $\mathbb{R}^n$, for $n \geq 3$, for which an entire conformal curve yields a conformal covering map. As a related result, we prove that if a non-constant entire conformal curve factors through an $n$-dimensional submanifold, then the submanifold is calibrated and conformally equivalent to a flat non-compact quotient of $\mathbb{R}^n$.
\end{abstract}

\maketitle

\section{Introduction}
\subsection{Background on conformal curves}
A continuous mapping $f \colon \Omega \rightarrow \mathbb{R}^n$ from a domain $\Omega \subset \mathbb{R}^n$ is a \emph{conformal mapping} if
\begin{equation}\label{eq:conformal}
    f \in \mathcal{C}^{1}( \Omega, \mathbb{R}^n )
    \quad\text{and}\quad
    \| Df \|^n = \star f^{*}\mathrm{vol}_{\mathbb{R}^n}
    \quad\text{everywhere};
\end{equation}
here $\|Df\|$ is the operator norm of the differential of $f$, $f^{*}\mathrm{vol}_N$ is the pullback of the Riemannian volume form $\mathrm{vol}_{ \mathbb{R}^n } = dx_1 \wedge \dots \wedge dx_n$, and $\star f^{*}\mathrm{vol}_{N}$ its Hodge star dual. A direct computation shows that $\star f^{*}\mathrm{vol}_{\mathbb{R}^n}$ is the Jacobian determinant of $f$.

Similarly, a continuous mapping $f \colon \Omega \rightarrow \mathbb{R}^{2n}$, for a domain $\Omega \subset \mathbb{R}^2$, is a \emph{holomorphic curve} if
\begin{equation}\label{eq:holomorphic}
    f \in \mathcal{C}^{1}( \Omega, \mathbb{R}^{2n} )
    \quad\text{and}\quad
    \| Df \|^2 = \star f^{*}\omega_{\mathrm{sym}}
    \quad\text{everywhere}
\end{equation}
where $\omega_{\mathrm{sym}} = \sum_{ j = 1 }^{ n } dx_{2j-1} \wedge dx_{2j}$ is the standard symplectic form. It follows that the right-hand side of \eqref{eq:holomorphic} can be expressed as the sum of the Jacobian determinants of minors of the differential of $f$. In fact, if we identify $\mathbb{R}^2 \simeq \mathbb{C}$ and $\mathbb{R}^{2n} \simeq \mathbb{C} \times \dots \times \mathbb{C} = \mathbb{C}^{n}$, \eqref{eq:holomorphic} holds for $f = ( f_1, \dots, f_n ) \colon \mathbb{C} \rightarrow \mathbb{C}^n$ if and only if the component functions are holomorphic.

Equations \eqref{eq:conformal} and \eqref{eq:holomorphic}, respectively, are generalizations of Cauchy--Riemann equations for holomorphic mappings in the complex plane, where the volume form $\mathrm{vol}_{\mathbb{R}^n}$ and the symplectic form $\omega_{\mathrm{sym}}$ play a special role. In this manuscript, we consider the class of conformal curves which satisfy Cauchy--Riemann equations relative to a constant-coefficient form $\omega \in \Lambda^{n}\mathbb{R}^m \setminus \{0\}$ for some $2 \leq n \leq m$. To this end, a continuous map $F \colon \Omega \rightarrow \mathbb{R}^m$, for an open set $\Omega \subset \mathbb{R}^m$, is a \emph{conformal $\omega$-curve} if
\begin{equation}\label{eq:quasiregular:curve}
    F \in \mathcal{C}^{1}( \Omega, \mathbb{R}^m )
    \quad\text{and}\quad
    ( \|\omega\| \circ F ) \| DF \|^n = \star F^{*}\omega
    \quad\text{everywhere}.
\end{equation}
For a smooth form $\omega \in \Omega^{n}( \mathbb{R}^m )$,
\begin{equation}\label{eq:comass:extremality}
    \| \omega \|( x )
    =
    \max
    \left\{
        \omega_x( v_1 \wedge \dots \wedge v_n )
        \colon
        ( v_1, \dots, v_n )
        \text{ orthonormal}
    \right\}
\end{equation}
is the \emph{comass} of $\omega$ at the point $x$. Harvey and Lawson \cite{Har:Law:82} coined the term \emph{calibration} to refer to a closed form whose comass is at most one everywhere. A calibration $\omega$ determines a distinguished class of minimal surfaces. Indeed, an $n$-dimensional $\mathcal{C}^2$-submanifold in $\mathbb{R}^m$ is \emph{$\omega$-calibrated} if every tangent space of the submanifold has an oriented orthonormal basis realizing the value one when evaluated against the form at the basepoint. A similar calibration condition can be formulated in the context of varifolds and integral currents, leading to $\omega$-calibrated varifolds and integral currents, respectively. One of the seminal contributions of \cite{Har:Law:82} was the discovery of constant coefficient calibrations such as the \emph{Special Lagrangian}, \emph{associative}, and \emph{Cayley} calibrations with rich families of calibrated varifolds and integral currents. The discovery was preceded by the works of Federer \cite{Fed:69} and Berger \cite{Ber:70} in the context of Kähler and quaterniotic geometries, respectively. We refer to the monographs \cite{Joy:07,Mor:09:book} for further background. The literature for (pseudo)holomorphic curves is vast, see the monograph \cite{McD:Sal:04} for an overview.

When considering the class defined by \eqref{eq:quasiregular:curve} for a non-zero constant-coefficient form, we may always normalize the comass \eqref{eq:comass:extremality} to one. Thus we exclusively consider constant-coefficient calibrations $\omega \in \Lambda^{n}\mathbb{R}^m$ for $2 \leq n \leq m$. In this case, \eqref{eq:quasiregular:curve} simplifies to the equality
\begin{equation}\label{eq:conformalcurve:calibration}
    F \in \mathcal{C}^{1}( \Omega, \mathbb{R}^m )
    \quad\text{and}\quad
    \| DF \|^n = \star F^{*}\omega
    \quad\text{everywhere}.
\end{equation}
Equations \eqref{eq:conformal} and \eqref{eq:holomorphic} provide archetypical examples of \eqref{eq:quasiregular:curve} and \eqref{eq:conformalcurve:calibration}. Conformal curves have also been studied under the name \emph{Smith maps} by Cheng--Karigiannis--Madnick and Iliashenko--Karigiannis \cite{Che:Kari:Mad:20,Ilia:Karig:23} based on the associative geometry and under the name \emph{calibrated curves} by Heikkilä--Pankka--Prywes in \cite{Hei:Pan:Pry:23}. The conformal curves form a subclass of the quasiregular curves introduced by Pankka in \cite{Pan:20} where the equality \eqref{eq:quasiregular:curve} (resp. \eqref{eq:conformalcurve:calibration}) is relaxed to a two-sided comparability.

\subsection{Energy growth and curves}
The author and Pankka recently investigated conformal curves \cite{Iko:Pan:24} in the context of the Liouville theorem (for conformal mappings). Recall that Liouville's theorem states that for $3 \leq n$, any conformal mapping $f \colon \Omega \rightarrow \mathbb{R}^n$ from a domain $\Omega \subset \mathbb{R}^n$ is the restriction of a Möbius transformation. Analogously, we proved that for a $G_{\delta}$-dense subset of constant-coefficient calibrations $\omega \in \Lambda^n \mathbb{R}^m$, when $3 \leq n \leq m$, every conformal $\omega$-curve $\Omega \rightarrow \mathbb{R}^m$ from a domain $\Omega \subset \mathbb{R}^n$ factors through a Möbius transformation $\Omega \rightarrow \mathbb{R}^n$ and an affine map $\mathbb{R}^n \rightarrow \mathbb{R}^m$. In particular, if $\Omega = \mathbb{R}^n$, such a curve is an affine map. This stemmed from the isoperimetric rigidity of conformal curves based on $\omega$-calibrated integral currents, Almgren's isoperimetric inequality for integral currents \cite{Alm:86}, and the interaction of the Cauchy--Riemann equation \eqref{eq:conformalcurve:calibration} with pushforward of integral currents (see \cite{Fed:69}). As a particular example of these interactions, for an entire conformal curves $\mathbb{R}^n \rightarrow \mathbb{R}^m$, the average energy
\begin{align}\label{eq:asysmpoticdegree:relative:omega}
    r
    \mapsto
    h(r)
    \coloneqq
    \aint{ B_{r}(x_0) } \|DF\|^n dx
    \text{ is non-decreasing (see \Cref{lemm:localrigid} below).}
\end{align}
To deduce \eqref{eq:asysmpoticdegree:relative:omega}, the sharp isoperimetric constant for integral currents is critical.

The author also proved in \cite{Iko:23} (see \cite[Proposition 7.1]{Iko:Pan:24} for the formulation below) a related modulus of continuity estimate 
\begin{equation}\label{eq:modulusofcontinuity}
    | F(x) - F(y) |
    \leq
    C(n)
    \left( \aint{ B_{2r}(x_0) }  \|DF\|^n dx \right)^{ \frac{1}{n} }
    |x-y|
    \quad\text{for $x, y \in B_{r}(x_0)$.}
\end{equation}
We refer to the limiting value $\lim_{ r \rightarrow \infty } h(r)$ in \eqref{eq:asysmpoticdegree:relative:omega} as the \emph{asymptotic growth} of $F$ (relative to $x_0 \in \mathbb{R}^n$). We say that the asymptotic growth is \emph{super-Euclidean} if the limit is infinite and \emph{bounded} otherwise. By \eqref{eq:asysmpoticdegree:relative:omega} and \eqref{eq:modulusofcontinuity}, bounded asymptotic growth is equivalent to a Lipschitz bound for a conformal curve. We also remark that there are conformal $\omega$-curves with infinite asymptotic growth in all dimensions $n \geq 2$. For instance, in \cite[Proposition 1.6]{Iko:Pan:24}, we considered a conformal curve $\mathbb{R}^n \rightarrow \mathbb{R}^{2n}$ for which $\|DF\|^n(x_1, \dots, x_n) = e^{n x_n}$ for $( x_1, \dots, x_n ) \in \mathbb{R}^n$ for $n \geq 2$, satisfying the Cauchy--Riemann equations relative to a Special Lagrangian calibration. In dimension two, simpler examples can be obtained by considering any holomorphic curve $\mathbb{C} \rightarrow \mathbb{C}^k$ with at least one non-affine component.

Having bounded asymptotic growth is a rigid property in all dimensions.
\begin{theorem}\label{thm:euclidean:to:affine}
For every calibration $\omega \in \Lambda^n \mathbb{R}^m$, every conformal $\omega$-curve $\mathbb{R}^n \rightarrow \mathbb{R}^m$ with bounded asymptotic growth is an affine linear map.
\end{theorem}
The main content of \Cref{thm:euclidean:to:affine} is $n \geq 3$. Indeed, since the complex differential of a holomorphic curve is a holomorphic curve, it easily follows from \eqref{eq:asysmpoticdegree:relative:omega} and \eqref{eq:modulusofcontinuity} and the Liouville's theorem  for holomorphic mappings that a holomorphic curve with bounded asymptotic growth is an affine map. The methods in higher dimensions are more elaborate and build upon an analysis of a blow-down of such a curve. Nevertheless, we give a unified treatment of \Cref{thm:euclidean:to:affine} for all dimensions.

By combining \Cref{thm:euclidean:to:affine} with the isoperimetric rigidity of conformal curves, non-affine conformal curves have a super-Euclidean growth in the following strong sense.
\begin{corollary}\label{cor:Euclidean:nonaffine}
For every calibration $\omega \in \Lambda^n \mathbb{R}^m$ and every non-affine conformal $\omega$-curve $F \colon \mathbb{R}^n \to \mathbb{R}^m$, the average energy in \eqref{eq:asysmpoticdegree:relative:omega} is unbounded and strictly increasing for large radii.
\end{corollary}

\begin{remark}\label{remark-growth-of-subharmonic-functions}
\Cref{thm:euclidean:to:affine} has an interpretation in potential analytic terms when $n \geq 3$. To elaborate on the topic, the author and Pankka \cite[Section 4.1]{Iko:Pan:24} proved that if $F \colon \mathbb{R}^n \to \mathbb{R}^m$ is a non-constant conformal $\omega$-curve for a calibration $\omega \in \Lambda^{n} \mathbb{R}^m$, then $\rho = \|DF\|^{ \frac{n-2}{2} }$ is subharmonic. By a theorem due to Greene--Wu \cite{greene-wu-1974-integrals-subharmonic-functions-on-manifolds-of-nonnegative-curvature}, it follows that $\rho \not\in L^{p}( \mathbb{R}^n )$ for $p \in [1,\infty)$. Interpreting \Cref{thm:euclidean:to:affine} in this context, we conclude that either $F$ is affine or $\|DF\| \not\in L^{p}( \mathbb{R}^n )$ for $p \in [(n-2)/2, \infty]$.

Similar analysis holds when $n = 2$ but instead $\rho = \log \|DF\|$ is subharmonic. In this context, \Cref{thm:euclidean:to:affine} yields that either $F$ is affine or $\rho$ is not bounded from above (where we recall that an entire subharmonic function bounded from above in $\mathbb{R}^2$ is constant by the Liouville's theorem for subharmonic functions).
\end{remark}
%
%\begin{remark}\label{remark-growth-of-curves-into-Hadamard-spaces}
%The knowledge of the sharp isoperimetric profile of the Euclidean space is crucial for the proof of \Cref{thm:euclidean:to:affine}. Therefore it is natural to ask if a similar dichotomy holds for entire conformal curves $\mathbb{R}^n \rightarrow M$ mapping into a Cartan--Hadamard manifold $M$, for instance. Under mild assumptions on the form defining the curve, a conformal curve $\mathbb{R}^n \rightarrow M$ is constant or has infinite energy by a recent result due to the author, see \cite[Theorem 1.13]{Iko:24}. Obtaining an explicit dependence between the energy growth and the isoperimetric profile would be of interest. Steps toward this direction was taken in the recent work \cite{broder-iliashenko-madnick-2025-hyperbolicity-and-schwarz-lemmas-in-calibrated-geometry}.
%\end{remark}
%
\subsection{A factorization through a Euclidean quotient}
To set the stage for the subsequent theorem, following Gromov \cite{Gro:07}, we say that $N$ is \emph{elliptic} if there exists a wrapping map $f \colon \mathbb{R}^n \rightarrow N$. Given a closed and oriented Riemannian $n$-manifold $N$, a Lipschitz map $f \colon \mathbb{R}^n \rightarrow N$ is a \emph{wrapping map} if 
\begin{align}\label{eq:asymptoticdegree}
    \limsup_{ r \rightarrow \infty }
        \frac{ 1 }{ r^n }
        \int_{ B_{r}(x) } f^{*}\vol_{N}
        >
        0
\end{align}
where $f^{*}\vol_{N}$ is the pullback of the volume form. We also recall the related notion of \emph{quasiregular ellipticity} coined by Bonk--Heinonen \cite{Bo:He:01}. We say that a closed and oriented Riemannian $n$-manifold is \emph{quasiregularly elliptic} if there exists a non-constant quasiregular map $f \colon \mathbb{R}^n \rightarrow N$. Quasiregular maps are generalizations of conformal mappings between equi-dimensional spaces which naturally occur in complex dynamics, elliptic PDEs, and piecewise linear geometries, for instance. In general such maps are not Lipschitz but have an asymptotic growth bound similar to \eqref{eq:asymptoticdegree} for an exponent $c(n) \neq n$ when the image is not a homology sphere such as a lens space, see \cite[Theorem 1.11]{Bo:He:01}. We refer the reader to the recent survey \cite{Manin:Prywes:24} by Manin and Prywes for further connections between these notions of ellipticity. We simply mention that de Rham cohomology of an elliptic and quasiregularly elliptic manifold admits a graded algebra embedding into the exterior algebra of $\mathbb{R}^n$, see \cite{berdkinov-guth-manin-2024-degrees-of-maps-and-multiscale-geometry,heikkila-pankka-2025-de-rham-algebras-of-closed-quasiregularly-elliptic-manifolds-are-euclidean}. See also \cite{heikkila-2024-quasiregular-curves-and-cohomology} for a related generalization for curves.

Based on the growth bound, \Cref{thm:euclidean:to:affine}, and \eqref{eq:modulusofcontinuity}, non-affine conformal curves $\mathbb{R}^n \to \mathbb{R}^{m}$ are not Lipschitz but do have super-Euclidean energy growth. Thus the energy growth of conformal curves share some features of the degree growth of wrapping and quasiregular maps, respectively. The theorem below concerns the factorization of a conformal curve through a submanifold and the imposed geometric constraints for the submanifold. Not only is the de Rham cohomology restricted as in the case of elliptic and quasiregularly elliptic manifolds but even the geometry of the space itself is highly constrained. More precisely, we have the following.
\begin{theorem}\label{thm:nonaffine:factorization}
If a non-constant conformal curve $\mathbb{R}^n \rightarrow \mathbb{R}^m$ factors through a connected smooth $n$-dimensional submanifold, the submanifold is a calibrated submanifold of $\mathbb{R}^m$ that is conformally equivalent to a non-compact quotient of $\mathbb{R}^n$. When $n \geq 3$, the curve is a conformal covering map of the submanifold and factors as a composition of a locally isometric covering map $\pi \colon \mathbb{R}^n \rightarrow N$ and a conformal embedding $g \colon N \rightarrow \mathbb{R}^m$ for a flat oriented manifold $N$.
\end{theorem}
When $n \geq 3$, one of the insights provided by \Cref{thm:nonaffine:factorization} is that the submanifold is conformally equivalent to an orbit space $\mathbb{R}^k \times \mathbb{T}_{n-k}/H$ where $H$ is a discrete subgroup of orientation-preserving isometries of $\mathbb{R}^k \times \mathbb{T}_{n-k}$ acting freely on $\mathbb{R}^k \times \mathbb{T}_{n-k}$ where $\mathbb{T}_{n-k}$ is the $(n-k)$-dimensional torus. This is based on the classification of non-compact oriented flat Riemannian spaces, see e.g. \cite[Chapter 8]{ratcliffe-2019-foundations-of-hyperbolic-manifolds}. We emphasize that the submanifold cannot be conformally equivalent to a quotient of the $n$-dimensional torus or to a quotient of a sphere by the non-compactness. The non-compactness can be deduced in many ways. For example, it follows from the calibration property, the Liouville's theorem for quasiregular curves \cite{Pan:20}, or from Caccioppoli's inequality, see \Cref{remark-caccioppoli} below.

The theorem above is sharp in two ways. First, the conformal curve in \cite[Proposition 1.6]{Iko:Pan:24} factors through $\mathbb{R} \times \mathbb{T}_{n-1}$ for the $(n-1)$-dimensional torus $\mathbb{T}_{n-1}$, so such a submanifold is not necessarily simply connected. Second, if $n = 2$, the complex exponential $z \mapsto e^{z}$ factors through the submanifolds $\mathbb{C} \setminus \{0\}$ or $\mathbb{C}$, so the covering map conclusion holds only when $n \geq 3$.

Since \Cref{thm:nonaffine:factorization} is independent of the other parts of the manuscript, we present its proof in \Cref{sec:factorization:submanifold}. In \Cref{sec:growthlemma}, we formulate the key growth result needed for \Cref{thm:euclidean:to:affine} and \Cref{cor:Euclidean:nonaffine} while the proofs are in \Cref{sec:mainproofs}.

\section{Factorization through a submanifold}\label{sec:factorization:submanifold}

We first recall two basic facts from the theory of minimal surfaces and conformal geometry, respectively.
\begin{lemma}\label{lemma:noncompact}
If $N \subset \mathbb{R}^m$ is a closed submanifold, then $N$ is not a minimal surface.
\end{lemma}
There are several ways to see \Cref{lemma:noncompact}. For instance, when $N$ is minimal, the coordinate functions of the inclusion $N \xhookrightarrow{} \mathbb{R}^m$ are harmonic, and thus constant by the maximum principle if $N$ were closed. See \cite[Chapter I, Corollary 10]{Lawson:80} for an alternate proof.

The second basic fact we need is the following formulation of Picard's theorem in higher dimensions.
\begin{lemma}\label{lemma:picard}
When $n \geq 3$ and $\widetilde{N}$ is a simply connected manifold without boundary, a non-constant conformal map $\widetilde{f} \colon \mathbb{R}^n \rightarrow \widetilde{N}$ is a conformal diffeomorphism onto $\widetilde{N}$ or $\widetilde{N}$ is conformally equivalent to the $n$-sphere and the image of $\widetilde{f}$ omits a point.
\end{lemma}
The complex exponential illustrates that the theorem is false when $n = 2$.
\begin{proof}
We first recall that conformal homeomorphisms between Riemannian manifolds are conformal diffeomorphisms by Ferrand's theorem \cite[Theorem A]{Lel:Fer:76}.

The following argument is based on the proof of \cite[Proposition 1.4]{Bo:He:01}. We first recall that $\widetilde{f}$ is a local homeomorphism due to \cite[Theorem II.10.5, p. 232]{Ric:93}. Then, by Zorich's theorem, $\widetilde{f}$ is a homeomorphism from $\mathbb{R}^n$ onto its image and the complement of the image in $\widetilde{N}$ is zero-dimensional, see \cite[p.336]{Gro:07} for this formulation. Since homeomorphisms preserve the number of ends, it follows that the complement of the image is at most a point. If the image omits a point, then $\widetilde{f}$ extends to a conformal homeomorphism from $\mathbb{S}^n \simeq \mathbb{R}^n \cup \{\infty\}$ onto $\widetilde{N}$ as point singularities have zero $n$-capacity. By Ferrond's result above, the extension is a conformal diffeomorphism. This implies that $\widetilde{N}$ is conformally diffeomorphic to the $n$-sphere $\mathbb{S}^n$. In case the image of $\widetilde{f}$ does not omit any points, then $\widetilde{f}$ is a conformal diffeomorphism onto $\widetilde{N}$. The claim follows.
\end{proof}

\begin{proof}[Proof of \Cref{thm:nonaffine:factorization}]
We consider a conformal $\omega$-curve $F \colon \mathbb{R}^n \to \mathbb{R}^m$ factoring through a submanifold $N \subset \mathbb{R}^m$.

We first consider the case $n \geq 3$. Let $f$ denote the corestriction of $F$ to $N$. The map $f$ is conformal and lifts to a conformal map $\widetilde{f} \colon \mathbb{R}^n \rightarrow \widetilde{N}$ where $\widetilde{N}$ is the Riemannian universal cover of $N$. By \Cref{lemma:picard}, either $\widetilde{f}$ is a conformal diffeomorphism onto $\widetilde{N}$ or $\widetilde{N}$ and $\mathbb{S}^n$ are conformally diffeomorphic. In the latter case, $\widetilde{N}$ is conformally diffeomorphic to $\mathbb{S}^n$, $f$ has full-rank everywhere, and its image excludes at most a point in $N$. Thus $N$ is $\omega$-calibrated by the continuity of the Gauss map of $N$. Since $N$ is closed as a quotient of $\widetilde{N}$, we obtain a contradiction with \Cref{lemma:noncompact} (see also \Cref{remark-caccioppoli} below). So only the former case is possible. As $f$ has full-rank at each point and maps onto the submanifold, the submanifold is $\omega$-calibrated by \cite[Theorem 1.1]{Iko:Pan:24}. The non-compactness of $N$ follows from \Cref{lemma:picard} (see also \Cref{remark-caccioppoli} below). Since $f$ is a conformal covering map of $N$, its covering group $G$ forms a subgroup of the (orientation-preserving) conformal automorphisms of $\mathbb{R}^n$ such that $G$ acts properly discontinuously and freely on $\mathbb{R}^n$. Then $( x \mapsto g(x) = \lambda A(x) + x_0 )$ for $\lambda \in (0,\infty)$, $A \in \mathrm{SO}(n)$, and $x_0 \in \mathbb{R}^n$ when $g \in G$. As the subgroup generated by $g \in G \setminus \{e\}$ acts properly discontinuously and freely on $\mathbb{R}^n$, it readily follows that $\lambda = 1$, i.e. that $g$ is an orientation-preserving isometry of $\mathbb{R}^n$. As $f$ is invariant under the action of $G$, it holds that $f$ descends to the orbit space $\mathbb{R}^n / G$. This finishes the proof of the case $n \geq 3$.

Next, we consider the case $n = 2$. We claim that $N$ is conformally equivalent to $\mathbb{R}^2$ or $\mathbb{R}^2 \setminus \{0\}$. To see this, by the uniformization theorem and Liouville's theorem, the Riemannian universal cover of $\widetilde{N}$ is conformally equivalent to $\mathbb{R}^2$ or $\mathbb{S}^2$. In either case, by the Picard theorem, the lift $\widetilde{f}$ of $f$ omits at most two points of $\widetilde{N}$, and as the branch set of $f$ is discrete, it follows that $N$ is $\omega$-calibrated outside a discrete set. However, by continuity of the Gauss map of $N$, $N$ is $\omega$-calibrated and thus minimal. Observe that $N$ is oriented by $\omega$ and non-compact by \Cref{lemma:noncompact}. Thus $N$ is conformally equivalent to $\mathbb{R}^2$ or $\mathbb{R}^2\setminus\{0\}$. By Picard's theorem, the map $f$ omits at most a point in the former case while it must map onto in the latter case.
\end{proof}
\begin{remark}\label{remark-caccioppoli}
The noncompactness of $N$ in \Cref{thm:nonaffine:factorization} has a more direct proof. Indeed, it holds that
\begin{align}\label{equation-caccioppoli}
    \left( \aint{ B_{r}(x_0) } \|DF\|^n \,dy \right)^{1/n} \leq C(n) \frac{ \diam( F( B_{2r}(x_0) ) ) }{ 2r }
\end{align}
by the Caccioppoli inequality for conformal curves, see e.g. \cite[Proposition 7.1]{Iko:Pan:24}. Thus, if $F$ is entire and non-constant, it holds that $r \mapsto \diam( F( B_{r}(x_0) ) )$ grows at least linearly as $r \rightarrow \infty$. We note that in \cite[Proposition 1.6]{Iko:Pan:24}, the diameter $\diam( F( B_{r}(x_0) ) )$ grows exponentially as $r \rightarrow \infty$.

Combining \eqref{eq:modulusofcontinuity} with \eqref{equation-caccioppoli} implies that an entire conformal curve $F \colon \mathbb{R}^n \to \mathbb{R}^m$ has super-Euclidean growth if and only if $r \mapsto \diam F( B_{r}(x_0) )$ grows super-linearly.
\end{remark}

\section{A growth lemma}\label{sec:growthlemma}
We start with the proof of the first half of \Cref{cor:Euclidean:nonaffine} as this will be important for the proof of \Cref{thm:euclidean:to:affine}.
\begin{lemma}\label{lemm:localrigid}
Let $2 \leq n \leq m$ and let $F \colon \mathbb{R}^n \rightarrow \mathbb{R}^m$ be a conformal $\omega$-curve for a calibration $\omega \in \Lambda^n \mathbb{R}^m$. Let $x_0 \in \mathbb{R}^n$. Then
\begin{align*}
    r \mapsto h(r) = \aint{ B_{r}(x_0) } \|DF\|^n \,dx
\end{align*}
is non-decreasing. Moreover, if the derivative is zero at some $r \in \{ h >  0\}$, then $F$ is affine.
\end{lemma}
\begin{proof}
Since $F \in \mathcal{C}^{1}( \mathbb{R}^n, \mathbb{R}^m )$, it follows that $h$ is differentiable for $r > 0$. Furthermore, it holds that
\begin{align*}
    h'(r)
    &=
    \frac{ 1 }{ \omega_n r^{n} }
    \left(
        \int_{ \partial B_{r}(x_0) } \|DF\|^{n} \,d\mathcal{H}^{n-1}
        -
        \frac{n}{r}
        \int_{ B_{r}(x_0) } \|DF\|^{n} \,dx
    \right)
    \\
    &=
    \frac{ n }{ r }
    \left(
        \aint{ \partial B_{r}(x_0) } \|DF\|^{n} \,d\mathcal{H}^{n-1}
        -
        \aint{ B_{r}(x_0) } \|DF\|^{n} \,dx
    \right).
\end{align*}
By Hölder inequality,
\begin{align*}
    \aint{ \partial B_{r}(x_0) } \|DF\|^{n} \,d\mathcal{H}^{n-1}
    \geq
    \left( \aint{ \partial B_{r}(x_0) } \|DF\|^{n-1} \,d\mathcal{H}^{n-1} \right)^{ \frac{n}{n-1} }
\end{align*}
where \cite[Theorem 4.1]{Iko:Pan:24} yields that
\begin{align}\label{eq:intermediateinequality:almgren}
    \left( \aint{ \partial B_{r}(x_0) } \|DF\|^{n-1} \,d\mathcal{H}^{n-1} \right)^{ \frac{n}{n-1} }
    \geq
    \aint{ B_{r}(x_0) } \|DF\|^{n} \,dx.
\end{align}
Therefore $h'(r) \geq 0$. It follows that $h$ is non-decreasing. In case $h'(r) = 0$, then \eqref{eq:intermediateinequality:almgren} holds with an equality. Combining this with \cite[Theorem 4.1]{Iko:Pan:24} guarantees that the image $F( B(x_0,r) )$ is contained in an $n$-dimensional affine subspace of $\mathbb{R}^m$. The elementary factorization, \cite[Theorem 3.8]{Iko:Pan:24}, implies that $F|_{ B(x_0,r) }$ factors through a conformal map $g \colon B( x_0, r ) \rightarrow \mathbb{R}^n$ and an affine isometry $L \colon \mathbb{R}^n \rightarrow \mathbb{R}^m$. When $n \geq 3$, by Liouville's theorem, there exists a Möbius transformation $G \colon \mathbb{R}^n \cup \{\infty\} \rightarrow \mathbb{R}^n \cup \{\infty\}$ that restricts to $g$. Moreover, in case $h(r) > 0$, it follows that $G$ is non-constant. Then the continuation principle, \cite[Corollary 4.5]{Iko:Pan:24}, implies that $F = L \circ G|_{ \mathbb{R}^n }$. Since entire Möbius transformations are affine, it follows that $F$ is affine as claimed. When $n = 2$ and $h(r) > 0$, we have that $g$ is nonconstant. Since $\|DF\| = \|Dg\|$, the conclusion $h'(r) = 0$ leads to
\begin{align*}
    \left( \aint{ \partial B_{r}(x_0) } \|Dg\| \,d\mathcal{H}^{1} \right)^2
    =
    \aint{ \partial B_{r}(x_0) } \|Dg\|^2 \,d\mathcal{H}^{1}
    =
    \aint{ B_{r}(x_0) } \|Dg\|^2 \,dx.
\end{align*}
It follows from Carleman's work, see \cite[p. 159]{Carl:1921}, that the equalities hold if and only if $\|Dg\|^2$ is a constant or equal to the Jacobian of a non-constant Möbius transformation. In particular, there exists a Möbius transformation $G \colon \mathbb{R}^2 \cup \{\infty\} \rightarrow \mathbb{R}^2 \cup \{\infty\}$ that restricts to $g$. Again, the continuation principle, \cite[Corollary 4.5]{Iko:Pan:24}, implies that $F = L \circ G|_{ \mathbb{R}^2 }$. It follows that $G$ fixes the infinity point and is therefore an affine Möbius transformation. The claim follows.
\end{proof}

\section{Proof of the main results}\label{sec:mainproofs}

Let $\delta^n(y) = \lim_{ r \rightarrow \infty } \aint{ B_r(y) } \|DF\|^n \, dx$. Suppose that $\delta = \delta(x_0) < \infty$ for some $x_0 \in \mathbb{R}^n$. Our aim is to prove that $F$ is affine. We first observe that $\delta(y) = \delta$ for every $y \in \mathbb{R}^n$. Indeed, it holds that
\begin{align*}
    B_{ r - |x_0-y| }(x_0) \subset B_{r}(y) \subset B_{ r + |x_0-y| }(x_0)
\end{align*}
and therefore 
\begin{align*}
    \lim_{ r \rightarrow \infty }
    \aint{ B_{r}(y) } \|DF\|^n \,dx
    =
    \delta^n
    \quad\text{for every $y \in \mathbb{R}^n$.}
\end{align*}
With this observation and \Cref{lemm:localrigid} at hand, it follows that $F$ is $\delta$-Lipschitz. In case $\delta = 0$, it follows that $F$ is constant and the proof is complete. The remaining section concerns the case $\delta \in (0,\infty)$. We assume $\delta = 1$ from now on by considering $F/\delta$ in place of $F$.

\begin{lemma}\label{lemm:blowdowns:linear}
Whenever $( y_j, r_j )_{ j \geq 1 }$ are such that $r_j \rightarrow \infty$ and
\begin{align*}
    1 = \lim_{ j\to\infty } \aint{ B_{r_j}(y_j) } \|DF\|^n \,dx,
\end{align*}
then a subsequence of $( F_j(\cdot) = ( F(y_j + r_j \cdot ) - F(y_j) )/r_j )_{ j \geq 1 }$ converge uniformly on compact subsets of $\mathbb{R}^n$ to a linear isometry $G \colon \mathbb{R}^n \rightarrow \mathbb{R}^m$.
\end{lemma}
Since \Cref{lemm:blowdowns:linear} implies that the blow-downs of $F$ are linear isometries, the curve is proper, i.e. $F^{-1}(K)$ is compact whenever $K \subset \mathbb{R}^n$ is compact, and has the following geometric growth at infinity.
\begin{corollary}\label{cor:properness}
Let $x_0 \in \mathbb{R}^n$. Then, for every $r > 0$, consider the maximal and minimal radii $s_r \in (0,\infty]$ and $S_r \in (0,\infty]$, respectively, for which
\begin{align*}
    B_{s_r}(x_0) \subset F^{-1}( B_{r}(F(x_0)) ) \subset B_{S_r}(x_0).
\end{align*}
Then the radii satisfy $r \leq s_r \leq S_r$ and $\lim_{ r \rightarrow \infty } S_r/r = 1$. In particular, the map $F$ is proper.
\end{corollary}
\begin{proof}[Proof of \Cref{lemm:blowdowns:linear}]
Observe that $\| DF_j \|(y) \leq 1$ for $y \in \mathbb{R}^n$, so $( F_j )_{ j \geq 1 }$ is a sequence of $1$-Lipschitz maps that satisfy $F_j(0) = 0$. It follows from Arzelà--Ascoli theorem that a subsequence $( F_{j_i} )_{ i \geq 1 }$ converges uniformly in compact subsets of $\mathbb{R}^n$ to a $1$-Lipschitz map $G \colon \mathbb{R}^n \rightarrow \mathbb{R}^m$.

The weak convergence of $( F_{j_i}^{*}\omega )_{ i \geq 1 }$ to $G^{*}\omega$ in $L^{1}_{loc}( \mathbb{R}^n )$ follows from the locally uniform convergence by standard weak continuity of the minors of the differentials; see e.g. \cite[Lemma 5.10]{Rindler:18}. The fact that $G$ is a conformal $\omega$-curve follows, for instance, from \cite[Proposition 7.2]{Iko:Pan:24}. So \eqref{eq:conformalcurve:calibration} and the normalization of $( F_j )_{ j \geq 1}$ imply that
\begin{align*}
    \aint{ B_{1}(0) } \|DG\|^n \,dx
    =
    \lim_{ i\to\infty } \aint{ B_{r_{j_i}}(y_{j_i}) } \|DF\|^n \,dx
    =
    1.
\end{align*}
Since $G$ is $1$-Lipschitz, it also holds that
\begin{align*}
    \aint{ B_{s}(0) } \|DG\|^n \,dx
    \leq
    1
    \quad\text{for $s > 1$,}
\end{align*}
so $G$ is affine by \Cref{lemm:localrigid}. Since $G$ is an affine conformal curve and the average energy in a ball is one, the map $G$ is an affine isometry.
\end{proof}

\begin{proof}[Proof of \Cref{cor:properness}]
The $1$-Lipschitz property of $F$ implies that $r \leq s_{r} \leq S_{r}$ for every $r > 0$. Let $( r_i )_{ i \geq 1 }$ be a positive sequence tending to $\infty$ such that the ratio $S_{r_i}/r_i$ tends to $\limsup_{r \rightarrow \infty} S_{r}/r$. If we prove that the ratio tends to one, the conclusion $\lim_{ r \rightarrow \infty } S_{ r }/r = 1$ follows. This further implies that $F$ is proper.

Let $T_i = 2 \cdot r_i$ in case $S_{r_i} = \infty$ and $T_i = \frac{ i - 1 }{ i }S_{r_i}$ otherwise. Then consider $x_i \in F^{-1}( B_{ r_i }( F(x_0) ) ) \setminus B_{ T_i }(x_0)$ and $F_{i}(x) = ( F( x_0 + |x_i-x_0| x ) - F( x_0 ) )/ |x_i-x_0|$ for $x \in \mathbb{R}^n$ and $i \geq 1$.

Consider a subsequence such that $|x_{i_j}-x_0|/r_{i_j}$ tends to an accumulation point $\lambda$ of $\{ |x_i-x_0| / r_i \colon i \in \mathbb{N} \}$. If we prove that the accumulation point is one, the conclusion of the claim follows. By the choice of $( T_i )_{ i \geq 1 }$ and $( x_i )_{ i \geq 1 }$, it holds that $\lambda \leq 1$, so it suffices to prove that $\lambda \geq 1$. We apply \Cref{lemm:blowdowns:linear}.

Observe that $F_{i_j}(0) = 0$ and that
\begin{align*}
    \aint{ B_{1}(0) } \|DF_{i_j}\|^n \,dx
    =
    \aint{ B_{ |x_{i_j}-x_0| }(x_0) } \|DF\|^n \,dx
    \rightarrow
    1.
\end{align*}
Then, up to passing to a subsequence and relabeling, it holds that $( F_{i_j} )_{ j \geq 1 }$ converge uniformly on compact sets to an isometry $G \colon \mathbb{R}^n \rightarrow \mathbb{R}^m$ and $( ( x_{i_j} - x_0)/ |x_{i_j}-x_0| )_{ j \geq 1 }$ converges to some $x \in \partial B_{1}(0)$. It follows that $1 = | G(x) | = \lim_{ j \to \infty } | F_{i_j}( ( x_{i_j} - x_0)/ |x_{i_j}-x_0| ) | \leq \lim_{ j \rightarrow \infty } r_{i_j} / | x_{i_j}-x_0 | = \lambda$, so the proof is complete.
\end{proof}

\begin{proof}[Proof of \Cref{thm:euclidean:to:affine}]
In case $F$ is constant, the claim is already clear. Otherwise we normalize $F$ so that $F(0) = 0$ and $\sup_{x\in \mathbb{R}^n} \|DF\|(x) = 1$. We recall from \Cref{cor:properness} that $F$ is proper.

Since $F$ is proper, the pushforward $T = F_\sharp[ \mathbb{R}^n ]$ is a well-defined boundedly finite integral cycle that is $\omega$-calibrated. In fact, if $\Omega \subset \mathbb{R}^m$ is precompact and open, the mass measure $\|T\|$ satisfies
\begin{align*}
    \|T\|( \Omega )
    =
    \int_{ F^{-1}( \Omega ) } F^{*}\omega
    =
    \int_{ F^{-1}( \Omega ) } \|DF\|^n \,dx
\end{align*}
by \eqref{eq:conformalcurve:calibration} and the definition of the pushforward. In particular, $T$ is an area-minimizing integral cycle in the sense of boundedly-finite integral currents. By the monotonicity formula for $T$, see e.g. \cite[5.4.5]{Fed:69}, and the definition of $S_r$, we have
\begin{align*} 
    \omega_n r^n
    \leq
    \|T\|( B_{r}(F(x_0)) )
    &=
    \int_{ F^{-1}( B_{r}(F(x_0)) ) } \|DF\|^n \,dx
    \\
    &\leq
    \int_{ B_{S_r}(x_0) } \|DF\|^n \,dx.
\end{align*}
Therefore
\begin{align*}
    1
    \leq
    \frac{ \|T\|( B_{r}(F(x_0)) ) }{ \omega_n r^n }
    \leq
    \left( \frac{ S_r }{ r } \right)^n
    \aint{ B_{S_r}(x_0) } \|DF\|^n \,dx
\end{align*}
where the right-hand side converges to one as $r \rightarrow \infty$. Since
\begin{align*}
    r \mapsto \frac{ \|T\|( B_{r}(F(x_0)) ) }{ \omega_n r^n }
\end{align*}
is non-decreasing by the monotonicity formula (see e.g. \cite[5.4.3 (2)]{Fed:69}), it follows that $\|T\|( B_{r}(F(x_0)) ) = \omega_n r^n$ for every $x_0 \in \mathbb{R}^n$ and $r > 0$. We conclude that $\mathrm{spt}(T)$ is an affine subspace by \cite[5.4.5 (3)]{Fed:69}. The properness and the area formula for $F$ combined with \eqref{eq:conformalcurve:calibration} imply that $\mathrm{spt}(T) = F( \mathbb{R}^n )$. Therefore the image of $F$ is affine. The elementary factorization, \cite[Theorem 3.8]{Iko:Pan:24}, implies that $F$ factors through an entire conformal map $f \colon \mathbb{R}^n \rightarrow \mathbb{R}^n$ and an affine isometry $\mathbb{R}^n \rightarrow \mathbb{R}^m$. In case $n \geq 3$, it follows that $f$ is a Möbius transformation by the Liouville theorem. In case $n = 2$, we may use the fact that the complex differential of $f$ is a bounded holomorphic function and thus constant by Liouville's theorem. It follows that $F$ is affine in either case.
\end{proof}

\begin{proof}[Proof of \Cref{cor:Euclidean:nonaffine}]
\Cref{thm:euclidean:to:affine} implies that $F$ has super-Euclidean growth (relative to any $x_0 \in \mathbb{R}^n$). Combining this with \Cref{lemm:localrigid} implies that $h$ is strictly increasing in $\left\{ h > 0 \right\}$ and nondecreasing everywhere. The conclusion follows.
\end{proof}

\bibliographystyle{alpha}
\bibliography{Bibliography}

\end{document}